\documentclass{amsart}

\usepackage{amsmath}

\allowdisplaybreaks
\usepackage{setspace}
\usepackage{hyperref}
\usepackage{cleveref}
\usepackage{amsfonts}
\usepackage[all]{xy}
\usepackage{amssymb}
\DeclareMathOperator{\Ima}{Im}
\newcommand{\Aut}{{\rm Aut}}
\usepackage{xcolor}
\usepackage{graphicx} % Required for inserting images

\usepackage{xfrac}
\usepackage{amsfonts, bm}

\newcommand{\N}{{\mathbb N}}

\newcommand{\Z}{{\mathbb Z}}
\newcommand{\R}{{\mathbb R}}

\newcommand{\G}{{\mathcal G}}
\newcommand{\K}{{\mathcal K}}
\newcommand{\Hm}{{\mathbb H}}
\newcommand{\D}{{\mathfrak D}}
\newcommand{\Hk}{{\mathfrak H}}
\newcommand{\Hc}{{\mathcal H}}

\newcommand{\du}{{\mathrm d\,}}
\newcommand{\Ad}{\mathrm{Ad\,}}
\newcommand{\Id}{\mathrm{Id}}

\newcommand{\GL}{\mathrm{GL}}
\newcommand{\SL}{\mathrm{SL}}
\newcommand{\PSL}{\mathrm{PSL}}
\newcommand{\PGL}{\mathrm{PGL}}
\newcommand{\SO}{\mathrm{SO}}
\newcommand{\Fix}{\mathrm{Fix}}

\let\ol=\overline
\let\ap=\alpha

\let\mi=\setminus

\newtheorem{thm}{Theorem}[section]
\newtheorem{cor}[thm]{Corollary}
\newtheorem{lem}[thm]{Lemma}
\newtheorem{prop}[thm]{Proposition}

\newtheorem{example}[thm]{Example}
\newtheorem{remark}[thm]{Remark}

\begin{document}

\title[Twisted conjugacy classes in Lie groups]{TWISTED CONJUGACY CLASSES IN LIE GROUPS}

\author{Ravi Prakash}
\address{Ravi Prakash\\
School of Physical Sciences\\
Jawaharlal Nehru University\\
New Delhi 110 067}

\email{raviprakashg0@gmail.com}

\author{Riddhi Shah}\
\address{Riddhi Shah\\
School of Physical Sciences\\
Jawaharlal Nehru University\\
New Delhi 110 067}

\email{rshah@jnu.ac.in, riddhi.kausti@gmail.com}

\begin{abstract}
 We consider twisted conjugacy classes of continuous automorphisms $\varphi$ of a Lie group $G$. We obtain 
 a necessary and sufficient condition on $\varphi$ for its Reidemeister number, the number of twisted conjugacy classes,  
 to be infinite when $G$ is connected and solvable or compactly generated and nilpotent. We also show 
 for a general connected Lie group $G$ that the number of conjugacy classes is infinite. We prove that for 
 a connected non-nilpotent Lie group $G$, there exists $n\in\N$ such that Reidemeister number of $\varphi^n$ is infinite for every 
 $\varphi$. We say that $G$ has topological $R_\infty$-property if  the Reidemeister number of every $\varphi$ is infinite. We
obtain conditions on a connected solvable Lie group under which it has topological $R_\infty$-property; which, in particular, enables 
 us to prove that the group of invertible $n\times n$ upper triangular real matrices and its quotient group modulo its center have 
  topological $R_\infty$-property for every $n\ge 2$. We also prove that the Walnut group also has this property. We show that  
  $\SL(2,\R)$ and $\GL(2,\R)$ have topological $R_\infty$-property, and construct many examples of Lie groups with this property.
  
\end{abstract}
  
\maketitle
\noindent {\em 2020 Mathematics Subject Classification}. Primary: 22E15, 22D45. Secondary: 22E25.

\smallskip
\noindent{\bf Keywords:} Lie groups, Twisted conjugacy classes of automorphisms, Reidemeister number, Topological $R_\infty$-property.

\tableofcontents
\section{Introduction}

For an automorphism $\varphi$ of a group $G$, a relation 
 $\sim_\varphi$ on $G$ is defined as follows: for $x,y\in G$, $x \sim_\varphi y$, if $y=g x \varphi(g^{-1})$ for some $g \in G$. 
 Then $\sim_\varphi$ is an equivalence relation, and $x, y \in G$ are said to be 
 $\varphi$-twisted conjugate if $x\sim_\varphi y$. The equivalence classes with respect to this 
 relation are called $\varphi$-twisted conjugacy classes or Reidemeister classes of $\varphi$. 
 The $\varphi$-twisted conjugacy classes are the usual conjugacy classes if $\varphi$ is the 
 identity automorphism. We denote the $\varphi$-twisted conjugacy class containing $x \in G$
  by $[x]_\varphi$, and the set of all $\varphi$-twisted conjugacy classes by 
  $\mathcal{R}(\varphi):=\{[x]_\varphi\mid x \in G\}$. The cardinality of 
  $\mathcal{R}(\varphi)$, denoted by $R(\varphi)$, is called the Reidemeister number of $\varphi$. 
  A group $G$ is said to have $R_{\infty}$-property if $R(\varphi)=\infty$ for all $\varphi \in   \Aut(G)$.
  
Twisted conjugacy has its origin in Nielsen-Reidemeister fixed point theory 
and it also appears in other areas of Mathematics such as Algebraic geometry, Dynamical systems, 
Representation theory and Number theory; see Fel'shtyn \cite{F}, Fel'shtyn and Troitsky \cite{FT1}, Gon\c{c}alves 
et al \cite{GSaW} and the references cited therein. To determine which classes of groups have $R_{\infty}$-property 
is an active research area, which was initiated by Fel'shtyn and Hill \cite{FH}. We refer the reader to \cite{FT2, Se} and 
the references cited therein for more details.

Twisted conjugacy classes for abelian groups were studied by Dekimpe and Gon\c{c}alves in \cite{DeG}, and 
for nilpotent groups by Gon\c{c}alves and Wong  in \cite{GW}, and for polycyclic groups by Wong in \cite{W}.

For a (Hausdorff) topological group $G$, let $\Aut(G)$ denote the group of automorphisms of $G$ which are homomorphisms 
and homeomorphisms. If $G$ is any group without a topology, we assume that $G$ is endowed with the discrete topology, and 
any bijective map on $G$ is a homeomorphism, and in particular, any abstract automorphism of $G$ belongs to $\Aut(G)$.

In this paper, we study $R(\varphi)$ for continuous automorphisms $\varphi$ of (real) Lie groups and obtain conditions under which $R(\varphi)$ is finite 
or infinite. Twisted conjugacy classes in complex semisimple Lie groups appear in the work of Gantmakher \cite{Ga} in 1939; see Springer \cite{Sp}, 
see also Steinberg \cite{St}, and they have been studied by many mathematicians, see Mohrdieck and Wendt \cite{MoW}, 
Dekimpe and Pennincks \cite{DeP}, Lins de Araujo and Santos Rego \cite{LR1}, and the references cited therein. In recent years they have been studied 
on $\R^n$, and more generally on simply connected nilpotent Lie groups \cite{DeP}. Generalising some of the earlier results, we get necessary  and 
sufficient conditions for the set $\mathcal{R}(\varphi)$ to be finite or infinite for automorphisms $\varphi$ of connected solvable Lie groups as follows:

\medskip
\noindent{\bf \Cref{Solv}.}
\emph{Let $G$ be a connected solvable Lie group and $\G$ be the Lie algebra of $G$. Let $\varphi \in  \Aut(G)$ and $\du\varphi$ be the 
corresponding Lie algebra automorphism of $\G$. Then either $R(\varphi)=1$ or $R(\varphi)=\infty$, and the following 
statements are equivalent:
\begin{enumerate}
    \item[$(1)$] $R(\varphi) = \infty$.
    \item[$(2)$] $1$ is an eigenvalue of $\du\varphi$.
    \item[$(3)$] $|\Fix(\varphi)|=\infty$, where $\Fix(\varphi)$ is the set of fixed points of $\varphi$.
    \end{enumerate}}

\medskip
Using \Cref{Solv} and some known results on torsion-free finitely generated nilpotent groups, we get a characterisation for compactly 
generated nilpotent Lie groups as follows.

\medskip
\noindent{\bf \Cref{Cptgenilp}.}
 \emph{Let $G$ be a compactly generated nilpotent Lie group and $\varphi\in\Aut(G)$. Then the following holds:
$R(\varphi) = \infty\mbox{ if and only if }|\Fix(\varphi)|= \infty$.}

\medskip
For an automorphism $\varphi$ of a compactly generated solvable Lie group, $R(\varphi)=\infty$ if $|\Fix(\varphi)|=\infty$ 
(see \Cref{Cptgensolv}). As for a general connected Lie group $G$ which is not nilpotent, we have that $\varphi^n$ has infinitely many 
Reidemeister classes for every $\varphi\in\Aut(G)$, for some $n\in\N$. More generally, we have the following:

\medskip
\noindent{\bf \Cref{ConnLie}.}
\emph{Let $G$ be a nontrivial connected Lie group. Then the number of conjugacy classes in $G$ is infinite. Furthermore, if $G$ is not 
nilpotent, then there exists a subgroup $L$ of finite index in $\Aut(G)$ such that $R(\alpha)=\infty$ for every 
 $\alpha\in L$; in particular, there exists $n \in \N$ such that $R(\varphi^n) = \infty$ for every $\varphi\in\Aut(G)$.}

\medskip
A topological group $G$ is said to have {\it topological} $R_\infty$-{\it property} if $R(\varphi)=\infty$ for every automorphism 
(homomorphism and homeomorphism) $\varphi$ of $G$. 
Since any group can be endowed with the discrete topology, the definition of topological $R_\infty$-property is a natural extension  
of the definition of $R_\infty$-property. There has been considerable study of $R_\infty$-property for linear groups, and also of 
algebraic $R_\infty$-property, which is for morphisms on algebraic groups (see Nasybullov \cite{N1, N2}, 
Bhunia and Bose \cite{BB1, BB2} and the references cited therein).

Note that $\R^n$, $n\in\N$, and many nilpotent non-abelian groups do not have (topological) $R_\infty$-property (see \cite{N2}). 
As for connected solvable Lie groups $G$, we have a necessary and sufficient condition for $R(\varphi)$ to be infinity in \Cref{Solv} 
for $\varphi\in\Aut(G)$. So the question arises whether there are connected solvable Lie groups with this property, i.e.\ each of its 
(continuous) automorphism has infinite Reidemeister classes. The answer is affirmative for many solvable Lie groups. 
We obtain sufficient conditions on connected solvable Lie groups to have topological 
$R_\infty$-property (see \Cref{Solv-infty}). This enables us to prove the following:

\medskip
\noindent{\bf \Cref{upper}.} 
\emph{The group of invertible $n\times n$ upper triangular real matrices has topological $R_\infty$-property for every $n\ge 2$.}
\medskip

For the group as in \Cref{upper}, its quotient group modulo its center also has topological $R_\infty$-property (see \Cref{pbn}). The 
solvable (non-nilpotent) Lie groups $\R^*\ltimes\R$ and $\R\ltimes\R$ have topological $R_\infty$-property 
(see \Cref{axpb}). But $\mathbb{S}^1\ltimes\R^2$ and certain groups of the form $\R\ltimes\R^2$ do not have this property 
(see \Cref{non-ex}). We also show that the Walnut group has this property (see \Cref{walnut}).

In case of semisimple Lie groups, it is known that for $n\ge 3$, $\SL(n,\R)$ and $\GL(n,\R)$ have $R_\infty$-property (cf.\ \cite{N1}). We prove the 
remaining case $n=2$ in the following:

\medskip 
\noindent{\bf \Cref{SL2}.}
\emph{$\SL(2,\R)$ and $\GL(2,\R)$ have topological $R_\infty$-property.}

\medskip
As a consequence, we get some examples of general Lie groups with topological $R_\infty$-property e.g.\ 
$\widetilde{\SL(2,\R)}$, the universal cover of $\SL(2,\R)$ and also of $\PSL(2,\R)$, more generally see \Cref{PSL2}.

We may note here that discrete linear groups $\GL(n,\Z)$ as well as $\SL(n,\Z)$ and $\PSL(n,\Z)$ have $R_\infty$-property for 
$n\geq 2$; see Mubeena and Sankaran \cite{MuSa1}, see also \cite{MuSa2, MuSa3} for this property for some types of 
lattices and certain discrete subgroups in certain semisimple Lie groups.

\medskip
\noindent{\bf Notations:}  From \S 3 onwards, $G$ is a (real) Lie group and $\Aut(G)$ is the group of continuous automorphisms of $G$. 
Note that any abstract automorphism of a Lie group which is continuous is open, and hence a homeomorphism. Let $G^0$ denote 
the connected component of the identity $e$ in $G$, it is a closed characteristic subgroup of $G$. Let $\G$ denote the Lie algebra 
of a connected Lie group $G$, and let $\exp:\G\to G$ denote the exponential map which is continuous and its restriction to a small 
neighbourhood of $0$ in $\G$ is a homeomorphism onto a small neighbourhood of the identity in $G$. For $\varphi\in\Aut(G)$, there 
is a corresponding Lie algebra automorphism $\du\varphi$ such that $\exp\circ\,\du\varphi=\varphi\circ\,\exp$. The map 
$\du:\Aut(G)\to \GL(\G)$ is an injective homomorphism and its image is closed, and $\Aut(G)$ is isomorphic to its image in $\GL(\G)$. 
For $g\in G$, let $i_g$ denote the inner automorphism by $g$, and let $\Ad g=\du(i_g)$. Then $\Ad:G\to\GL(\G)$ is a (continuous) 
homomorphism. For a connected Lie group $G$, let $R$ denote the radical of $G$ and let $N$ denote the nilradical of $G$, where 
the radical $R$ (resp. the nilradical $N$) is the maximal connected solvable (resp.\ nilpotent) normal subgroup of $G$. Either $G$ is 
solvable or $G/R$ is semisimple (i.e.\ its radical is trivial). Any closed subgroup $H$ of $G$ is a Lie group with respect to the subspace 
topology on it. Let $[H,H]$ denote the commutator subgroup of $H$, and it is normal (resp.\ characteristic) in $G$ if $H$ is so. For 
a connected solvable Lie group $G$, the commutator subgroup $[G,G]$ is nilpotent (see e.g.\ \cite{MSR}). Any connected Lie group 
$G$ admits a maximal compact connected central subgroup $K$ which is contained in the nilradical $N$ and $G/K$ has no 
nontrivial compact connected central subgroups. Lie groups are locally compact and metrisable, and a connected locally compact Hausdorff 
group is a Lie group $G$ if it admits a neighbourhood of the identity $e$ which has no nontrivial compact subgroups in it. For a closed subgroup 
$H$ of $G$, if it is normal, the corresponding quotient group $G/H$ is also a Lie group. If $G$ and $H$ are connected, 
then $H$ has a corresponding Lie algebra $\Hc$ which is a subalgebra of the Lie algebra $\G$ of $G$, and in case $H$ is normal, 
then $\G/\Hc$ is the Lie algebra of $G/H$.  For the structure and properties of Lie groups,  we refer the reader to Hochschild \cite{H}. 
For subgroups $H$ and $L$ of $G$, let $[H,L]$ denote the subgroup generated by $\{hlh^{-1}l^{-1}\mid h\in H, l\in L\}$. For any set $A$, 
let $|A|$ denote the cardinality of the set $A$.

In \S 2, we review some useful results. In \S 3, we discuss results on solvable Lie groups and prove Theorems \ref{Solv} and \ref{Cptgenilp}. 
In \S 4, we discuss Reidemeister number for automorphisms of a general connected Lie group. In \S 5, we obtain sufficient conditions 
under which a  connected solvable Lie group has topological $R_\infty$-property in \Cref{Solv-infty} and prove \Cref{upper}. We also prove this 
property for $\SL(2,\R)$ and $\GL(2,\R)$ in \Cref{SL2}, and obtain several examples of solvable and non-solvable Lie groups with 
topological $R_\infty$-property.

\section{Preliminaries}

In this section, for a group $G$ with the identity $e$, and $\Aut(G)$ the group of all (abstract) automorphisms of $G$, 
we state some known results and other elementary results which will be useful. 
We first recall two elementary well-known statements.

\begin{lem} \label{inner} 
Let $G$ be a group and let $\varphi \in\Aut(G)$. For every $g\in G$, $R(i_g \circ \varphi)=R(\varphi)$, 
where $i_g$ is the inner automorphism by $g$. In particular, $R(i_g)=R(\Id)$, $g\in G$, where $\Id$ denotes the identity map on $G$. 
 \end{lem}

\begin{lem} \label{endo} 
Let $G$ be a group, $\varphi \in\Aut(G)$ and let $H$ be a $\varphi$-invariant normal subgroup of $G$.
Let  $\bar{\varphi}$ denote the automorphism of $G /H$ induced by $\varphi$. Then the following hold:
$R(\varphi) \geq R(\bar{\varphi})$. In particular, if $R(\bar{\varphi})=\infty$, then $R(\varphi)=\infty$.  
 If $H$ is characteristic in $G$ and $G /H$ has the $R_{\infty}$-property, then so does $G$.
\end{lem}

For $\varphi\in\Aut(G)$, let $\Fix(\varphi)=\{x\in G\mid \varphi(x)=x\}$, the set of fixed points of $\varphi$ in $G$, and $|\Fix(\varphi)|$ 
is the cardinality of $\Fix(\varphi)$.  Note that $\Fix(\varphi)$ is a subgroup of $G$. In the following lemma, statements (1) and (5) are known. 
We give a short proof for the sake of completeness. 
 
 \begin{lem}{\label{inv}} Let $G$ be a group and $\varphi \in \Aut(G)$. 
 Let $H$ be a normal $\varphi$-invariant subgroup of $G$.  Let $\varphi'= \varphi|_H$ and $\bar \varphi$ 
 be the automorphism of $G/H$ induced by $\varphi$. 
 Then the following hold:
 \begin{enumerate}
\item[{$(1)$}] If $R(\varphi')=\infty$ and $|\Fix(\bar\varphi)|<\infty$, then $R(\varphi) = \infty$. 
\item[{$(2)$}] If $R(\varphi')=1$, then $|\Fix(\bar\varphi)|\le |\Fix(\varphi)|$. 
\item[{$(3)$}] If $R(\bar\varphi)=1$, then $R(\varphi)\leq R(\varphi')$. In particular, if $R(\varphi')=R(\bar\varphi)=1$, then \hfill\break $R(\varphi)=1$.
\item[{$(4)$}] If $H$ is finite, then $R(\varphi)<\infty$ if and only if $R(\bar\varphi)<\infty$.
\item[{$(5)$}] If $H$ is central in $G$ and $R(\varphi')<\infty$ and $R(\bar\varphi)<\infty$, then $R(\varphi)<\infty$. 
\end{enumerate}
\end{lem}

\begin{proof} Let $\pi:G\to G/H$ be the natural projection and let $\bar{y}:=\pi(y)$, $y\in G$.

\smallskip
\noindent{\bf (1):} Suppose $R(\varphi')=\infty$. Let $\{x_\alpha\}\subset H$ be such that 
$[x_\alpha]_{\varphi'}$ are infinitely many distinct $\varphi'$-twisted conjugacy classes in $H$.
If possible, suppose $R(\varphi)<\infty$. Then there exists $n$ distinct $\varphi$-twisted conjugacy classes for 
$y_1,\ldots, y_n$ in $G$. Thus $[x_\alpha]_\varphi=[y_j]_\varphi$ for some fixed $j$ for infinitely many $\alpha$, i.e.\ 
$x_\alpha=g_\ap x_\beta\varphi(g_\alpha^{-1})$ for infinitely many $\alpha$, where $x_\beta$ is fixed. 
Since all $[x_\alpha]_{\varphi'}$ are distinct, we have that $g_\alpha\not\in H$ for infinitely many $\alpha$. Now 
$$
\bar{e}=\bar{x}_\alpha=\bar{g}_\ap\bar{x}_\beta\bar\varphi(\bar{g}_\alpha^{-1})=\bar{g}_\alpha\bar\varphi(\bar{g}_\alpha^{-1}).$$ 
Thus $\bar{g}_\ap\in \Fix(\bar\varphi)$ for infinitely many $\alpha$. As $[x_\alpha]_{\varphi'}$ are distinct, it follows that 
$\bar{g}_\ap$ are also distinct. This leads to a contradiction, hence $R(\varphi)=\infty$.

\smallskip
\noindent{\bf (2):} Suppose $R(\varphi')=1$. Let $g\in G$ be such that $\bar g\in\Fix(\bar\varphi)$, i.e.\ $\bar\varphi(\bar g)=\bar g$. 
Then $\varphi(g)=gx$ for some $x\in H$. As $R(\varphi')=1$, $H=[e]_{\varphi'}$, and there exists $a\in H$ such that 
$x=a\varphi(a^{-1})$. Then $\varphi(ga)=ga$, and hence $ga\in\Fix(\varphi)$. Thus $\pi(\Fix(\varphi))=\Fix(\bar\varphi)$, 
and hence $|\Fix(\bar\varphi)|\le|\Fix(\varphi)|$.

\smallskip
\noindent{\bf (3):} Now suppose $R(\bar\varphi)=1$. Let $x\in G$. Then $\pi(x)\in \pi(G)=[\bar e]_{\bar\varphi}$, and 
there exists $g\in G$ such that $\bar{x}=\bar{g}\bar\varphi(\bar{g}^{-1})$. As $H$ is normal in $G$, we get that 
$x\in g\varphi(g^{-1})H=gH\varphi(g^{-1})$. Therefore $x=gh\varphi(g^{-1})$ for some $h\in H$. Hence $[x]_\varphi=[h]_\varphi$ 
for some $h\in H$. Thus $R(\varphi)\le R(\varphi')$. The second assertion follows from the first.

\smallskip
\noindent{\bf (4):} Suppose $R(\bar\varphi)<\infty$. Let $x_1,\ldots, x_m\in G$ be such that $G/H=\cup_i [\bar x_i]_{\bar\varphi}$. 
As $H$ is finite, $H:=\{y_1,\ldots, y_n\}\subset G$. As $H$ is normal in $G$, we have that $G=\cup_{i,j}[x_iy_j]_\varphi$. 
Therefore, $R(\varphi)\le mn<\infty$. The converse follows from \Cref{endo}.

\smallskip
\noindent{\bf (5):} Suppose $H$ is central in $G$. Suppose also that $R(\bar\varphi)=m<\infty$ and 
$R(\varphi')=n<\infty$. Let $x_1,\ldots, x_m\in G$ be such that $G/H=\cup_{i=1}^m [\bar x_i]_{\bar\varphi}$, 
and let $y_1,\ldots, y_n\in H$ be such that $H=\cup_{j=1}^n[y_j]_{\varphi'}$. Then for any $g\in G$, 
$\bar g\in [\bar{x}_i]_{\bar\varphi}$ for some $x_i$. Then 
$$
gH\subset [x_i]_\varphi H\subset  [x_i]_\varphi (\cup_{j=1}^n[y_j]_{\varphi'})\subset \cup_{j=1}^n[x_iy_j]_\varphi,$$ 
as $H$ is central in $G$. Therefore, $G=\cup_{i,j} [x_iy_j]_\varphi$. Hence, $R(\varphi)\le mn <\infty$. 
\end{proof}

If $G$ is a compact connected abelian Lie group, i.e.\ $G = {\mathbb{T}}^n = ({{\mathbb{S}}^1})^n$ for some $n\in\N$, and 
if $\Aut(G)$ consists of continuous automorphisms of $G$, then 
$\GL_n(\mathbb{Z})$ can be identified with $\Aut(G)$ via $A=\left(a_{i j}\right) \mapsto \varphi_A$; 
where $\varphi_A\in\Aut(G)$ is defined as follows: 
$\varphi_A((t_1, \ldots, t_n))=(s_1, \ldots, s_n)$, with $s_i=\prod_{j=1}^n t_j^{a_{i j}}$, for all  $t_i \in \mathbb{S}^{1}$, $1 \leq i \leq n$.

We recall Theorem 2.2 of \cite{BB2}, which is valid for any continuous automorphism of a torus and will be useful for extending some results 
on simply connected nilpotent Lie groups to connected nilpotent Lie groups.

\begin{thm}{\rm\cite[Theorem 2.2]{BB2}} \label{Torus}
Let $G$ be an $n$-dimensional torus ${\mathbb{T}}^n$ for some $n\in\N$. Then the following are equivalent for any $\varphi \in  \Aut(G)$.
\begin{enumerate}
    \item $R(\varphi)=1$.
    \item $\det(A-\Id) \neq 0$, where $A \in \GL_n(\Z)$ is such that $\varphi=\varphi_A$.
    \item $\Fix(\varphi)$ is finite.
\end{enumerate}
\end{thm}

\section{Twisted conjugacy classes in connected solvable Lie groups}

In this section we discuss Reidemeister number of any (continuous) automorphism of a connected, or more generally, compactly generated 
solvable Lie group. We first get equivalent conditions on the automorphism $\varphi$ for $R(\varphi)$ to be finite or infinite. From now on 
$\Aut(G)$ is the set of continuous automorphisms of a Lie group $G$. 
The following is essentially known for simply connected nilpotent Lie groups.

\begin{thm} \label{Solv}
Let $G$ be a connected solvable Lie group and $\G$ be the Lie algebra of $G$. Let $\varphi \in  \Aut(G)$ and $\du\varphi$ be the 
corresponding Lie algebra automorphism of $\G$. Then either $R(\varphi)=1$ or $R(\varphi)=\infty$, and the following 
statements are equivalent:
\begin{enumerate}
    \item[{$(1)$}] $R(\varphi) = \infty$.
    \item[{$(2)$}] $1$ is an eigenvalue of $\du\varphi$.
    \item[{$(3)$}] $|\Fix(\varphi)|=\infty$, where $\Fix(\varphi)$ is the set of fixed points of $\varphi$.
    \end{enumerate}    
\end{thm}

\begin{proof} Let $G$ be as in the hypothesis and $\varphi\in\Aut(G)$.

\smallskip
\noindent{\bf Step 1:} First suppose $G$ is simply connected and nilpotent. Then by Lemmas 3.3 and 3.4 of \cite{DeP}, $R(\varphi)=1$, otherwise 
$R(\varphi)=\infty$ and (1) and (2) above are equivalent. Since $\G$ is a vector space, (2) is equivalent to the statement that 
$|\Fix(\du\varphi)|=\infty$; this is equivalent to the statement that 
$|\Fix(\varphi)|=\infty$ as $\exp\circ\,\du\varphi=\varphi\circ\,\exp$, and $\exp$ is a diffeomorphism. Thus (1-3) are equivalent in 
case $G$ is simply connected and nilpotent. We may note here that $R(\varphi)=1$ if and olny if  
$\Fix(\varphi)=\{e\}$, as any simply connected nilpotent group is torsion-free.

\smallskip
\noindent{\bf Step 2:} Now suppose $G$ is compact. Then $G$ is abelian and isomorphic to an $n$-dimensional torus. 
By \Cref{Torus} (see \cite[Theorem 2.2]{BB2}), either $R(\varphi)=1$ or $|\Fix(\varphi)|=\infty$. Now suppose $R(\varphi)\ne 1$.  
Then (3) holds. Let $V$ (resp.\ $U$) be a neighbourhood of 0 in $\G$ (resp.\ of the identity $e$ in $G$), such that $\exp:V\to U$ 
is a diffeomorphism. Let $W\subset V$ be an open neighbourhood of $0$ in $\G$ such that $\du\varphi(W)\subset V$. Then 
$\exp(W)\subset U$ is open in $G$. Since $|\Fix(\varphi)|=\infty$, $\exp(W)\cap \Fix(\varphi)\ne\{e\}$. Let $x\in \exp(W)\mi \{e\}$ be 
such that $\varphi(x)=x$. Then $\du\varphi(X)=X$, where $X\in W$ with $\exp X=x$. Hence $1$ is an eigenvalue of $\du\varphi$ 
and (2) holds. Thus $(3)\implies (2)$.

Now suppose (2) holds. Take $X\in \G\mi\{0\}$ such that  $\du\varphi(X)=X$, then 
$\varphi(\exp tX)=\exp tX$. Thus $(2)\implies (3)$. Now take $\rho:G\to G$ be such that $\rho(g)=g\varphi(g^{-1})$. 
Observe that $\Ima\rho$ is a compact connected subgroup of $G$. 
As $1$ is an eigenvalue of $\du\varphi$, $S:=\Ima(I-\du\varphi)$ is a proper subspace of $\G$, where 
$I$ is the identity matrix in $\GL(\G)$. As $G$ is abelian, $\G$ is an abelian Lie algebra and 
$S$ is a proper Lie subalgebra. Then we have that 
$$
\exp(S)=\{\exp(Y-\du\varphi(Y))\mid Y\in\G\}=\{(\exp Y)\varphi(\exp Y)^{-1}\mid Y\in\G\}=\Ima\rho.$$ 
Thus $\Ima\rho$ is a proper compact connected subgroup of $G$. Note that $\Ima\rho=[e]_{\varphi}$ as the exponential map 
is surjective. As $G$ is abelian, for the set $\mathcal{R}(\varphi)$ of Reidemeister classes of $\varphi$, we have that 
$\mathcal{R}(\varphi) = \{ [x]_\varphi \mid x \in G \} = \{x[e]_\varphi\mid x \in G\} = G/ \Ima\rho$. 
Now we have that $R(\varphi) = |\mathcal{R}(\varphi)| = |G / \Ima\rho| = \infty$. Thus (1) holds and we have that $(2)\implies (1)$. 
Also, since $(3)\implies (2)$, we get that $R(\varphi)=\infty$ in case $R(\varphi)\ne 1$.

Now suppose (1) holds, i.e.\ $R(\varphi)=\infty$. Then $R(\varphi)\ne 1$, hence by \Cref{Torus} (see \cite[Theorem 2.2]{BB2}), 
(3) holds, and hence $(1-3)$ are equivalent if $G$ is compact.

\smallskip
\noindent{\bf Step 3:} Now suppose $G$ is a connected nilpotent Lie group. Let $K$ be the maximal compact subgroup of $G$. 
Then $K$ is connected, central and characteristic in $G$, and $G/K$ is simply connected and nilpotent. 
Let $\varphi'=\varphi|_K$ and let $\bar\varphi$ be the automorphism of $G/K$ induced by 
$\varphi$. Now if $R(\bar\varphi)=\infty$, then $R(\varphi)=\infty$. If $R(\bar\varphi)<\infty$, we get from above that 
$R(\bar\varphi)=1$ and also $|\Fix(\bar\varphi)|=1$. Now if $R(\varphi')=\infty$, then by \Cref{inv}\,(1), $R(\varphi)=\infty$. 
Now suppose $R(\varphi')<\infty$. Then from Step 2, $R(\varphi')=1$ and by \Cref{inv}\,(3), we get $R(\varphi)=1$.

Now suppose (1) holds. If possible, suppose (2) does not hold, i.e.\ $1$ is not an eigenvalue of $\du\varphi$. Then 
1 is not an eigenvalue of $\du\varphi'$, and by Step 2, $R(\varphi')=1$. Also, $1$ is not an eigenvalue of $\du\bar\varphi$, where $\du\bar\varphi$ 
is the quotient Lie algebra automorphism of $\G/\K$, for the Lie algebra $\K$ of $K$ contained in $\G$. As $G/K$ is simply connected and 
nilpotent, by Step 1, $R(\bar\varphi)=1$. Now by \Cref{inv}\,(3), $R(\varphi)=1$, which leads to a contradiction. Thus (2) holds 
and we have that $(1)\implies (2)$. Now if (2) holds, then there exists a nonzero $X\in\G$ such that $\du\varphi(X)=X$, and we have that 
$\varphi(\exp tX)=\exp tX$, $t\in\R$, hence (3) holds. Thus $(2)\implies (3)$.

Now suppose (3) holds. Note that if $R(\bar\varphi)=\infty$, then $R(\varphi)=\infty$. Now suppose 
$R(\bar\varphi)<\infty$. Then $R(\bar\varphi)=1$ and from Step 1, $|\Fix(\bar\varphi)|=1$. Thus $\Fix(\varphi)\subset K$, and hence
$|\Fix(\varphi')|=\infty$. By Step 2 above, $R(\varphi')=\infty$. Hence by \Cref{inv}\,(1), $R(\varphi)=\infty$, i.e.\ (1) holds. Thus
$(1-3)$ are equivalent when $G$ is nilpotent.

\smallskip
\noindent{\bf Step 4:} Suppose $G$ is any connected solvable Lie group. Then $G_1:=\overline{[G, G]}$ is a closed normal 
characteristic Lie subgroup of $G$, and it is nilpotent (see e.g.\ \cite{MSR}). Now $G / G_1$ is a connected abelian 
Lie group.
 
 Let $\varphi'=\varphi|_{G_1}$ and let $\bar\varphi$ be the automorphism of $G/G_1$ induced by $\varphi$. If $R(\varphi')=1$ and 
 $R(\bar\varphi)=1$. Then $R(\varphi)=1$ by \Cref{inv}\,(3). Now suppose $R(\varphi)\ne 1$.  Then either  $R(\varphi')\ne 1$ or 
 $R(\bar\varphi)\ne 1$ (cf.\ \Cref{inv}\,(3)). As $G_1$ is nilpotent and $G/G_1$ is abelian, we get from Step 3 that $R(\varphi')=\infty$ or 
 $R(\bar\varphi)=\infty$. Suppose $R(\bar\varphi)=\infty$. Then by \Cref{endo}, $R(\varphi)=\infty$. Now suppose $R(\bar\varphi)<\infty$. Then 
 we have that $R(\varphi')=\infty$. Moreover, from Step 3, we have that $R(\bar\varphi)=1$ and $|\Fix(\bar\varphi)|<\infty$. Now by 
 \Cref{inv}\,(1), $R(\varphi)=\infty$. 
 
 Now suppose (1) holds, i.e.\ $R(\varphi)=\infty$. If $R(\bar\varphi)=\infty$, then by Step 3, $1$ is an eigenvalue of $\du\bar\varphi$, and
 hence $1$ is also an eigenvalue of $\du\varphi$, and (2) holds. If $R(\bar\varphi)<\infty$, then $R(\bar\varphi)=1$. Arguing as above, we
 have that $R(\varphi')=\infty$. Hence we get from Step 3 that $1$ is an eigenvalue of $\du\varphi'$, and hence, that of $\du\varphi$. 
 Thus (2) holds and we have that $(1)\implies (2)$. As seen in Step 3, $(2)\implies (3)$ is obvious. Now suppose (3) holds, 
 i.e.\ $|\Fix(\varphi)|=\infty$. Arguing as in Step 3, for $G$ and $G_1$ instead of $G$ and $K$ there, and using result for the nilpotent 
 case, we can show that $R(\varphi)=\infty$, and thus $(3)\implies (1)$, and hence $(1-3)$ are equivalent. 
\end{proof}

\begin{remark} We may note here that for an automorphism $\varphi$ on a connected solvable Lie group $G$, $R(\varphi)=1$ if and only if 
$|\Fix(\varphi)|<\infty$, and the latter statement is equivalent to the statement that $1$ is not an eigenvalue of $\du\varphi$.  
\end{remark}
 
The following corollary is easy to deduce from \Cref{Solv}.

\begin{cor} \label{multi} Let $G$ be a connected solvable Lie group. Let $\varphi\in\Aut(G)$ and let $H$ be a closed connected 
$\varphi$-invariant normal subgroup of $G$. Let $\varphi'=\varphi|_H$ and let $\bar\varphi\in\Aut(G/H)$ be the automorphism induced by $\varphi$.
Then the following hold:
\begin{enumerate}
\item $R(\varphi)=R(\varphi')R(\bar\varphi)$.
\item $R(\varphi)=1$ if and only if $R(\varphi')=R(\bar\varphi)=1$.
\item $R(\varphi)=\infty$ if and only if either $R(\varphi')=\infty$ or $R(\bar\varphi)=\infty$. 
\end{enumerate}
\end{cor}

Recall that a topological group $G$ is said to be compactly generated if there exists a compact subset of $G$ which generates the group $G$. 
Any connected locally compact group is compactly generated, in particular, any connected Lie group is compactly generated. 
A (not necessarily connected) Lie group $G$ is compactly generated if $G/G^0$ is finitely generated. The following is known for 
finitely generated torsion-free nilpotent groups  \cite[Theorem 2]{W}, and for connected nilpotent Lie groups, it follows from \Cref{Solv}.

\begin{thm}{\label{Cptgenilp}}
 Let $G$ be a compactly generated nilpotent Lie group and $\varphi\in\Aut(G)$. Then the following holds:
$R(\varphi) = \infty\mbox{ if and only if }|\Fix(\varphi)|= \infty$.
\end{thm}

\begin{proof} The assertion follows from \Cref{Solv} if $G$ is connected. Now we may assume that $G$ is not connected.

\smallskip
\noindent{\bf Step 1:} First suppose $G$ is a finitely generated (discrete) nilpotent group. It is known that $G$ has a unique maximal 
finite subgroup (say) $H$ (see \cite[Theorem 2]{Lo} or \cite[Lemma 3.1]{D2}). Hence $H$ is characteristic in $G$ and $G/H$ is torsion-free. 
By \Cref{inv}\,(4), $R(\varphi)=\infty$ if and only if $R(\bar\varphi)=\infty$. The latter statement is equivalent to the statement that 
$|\Fix(\bar\varphi)|=\infty$ \cite[Theorem 2]{W}, and hence it is equivalent to the statement that $|\Fix(\varphi)|=\infty$, as $H$ is finite.

\smallskip
\noindent{\bf Step 2:} Now suppose $G$ is not discrete. Let $G^0$ be the connected component of the identity $e$ in $G$. 
Then $G^0$ is a (connected) nilpotent Lie group and it is characteristic in $G$. In particular it is $\varphi$-invariant. Let $\varphi'=\varphi|_{G^0}$ 
and let $\bar\varphi$ be the automorphism of $G/G^0$ induced by $\varphi$.

First suppose $|\Fix(\varphi)|=\infty$. If $R(\bar\varphi)=\infty$, then $R(\varphi)=\infty$. 
Now suppose $R(\bar\varphi)<\infty$. Since $G/G^0$ is a (discrete) finitely generated nilpotent group, we get from Step 1, $|\Fix(\bar\varphi)|<\infty$. 
Then $|\Fix(\varphi')|=\infty$. As $G^0$ is nilpotent, by \Cref{Solv}, we get that $R(\varphi')=\infty$. Now by \Cref{inv}\,(1), 
we have that $R(\varphi)=\infty$.

\smallskip
\noindent{\bf Step 3:} Now suppose $|\Fix(\varphi)|<\infty$. We show that $R(\varphi)<\infty$. Then 
$|\Fix(\varphi')|<\infty$, and  by \Cref{Solv}$, R(\varphi')=1$. Now by \Cref{inv}\,(2), $|\Fix(\bar\varphi)|<\infty$. Then from Step 1, we get that 
$R(\bar\varphi)<\infty$. If $G^0$ is central in $G$, then by \Cref{inv}\,(5), $R(\varphi)<\infty$, and the proof is complete in this case.

Let $G_0=G^0$ and $G_m=\ol{[G,G_{m-1}]}$, $m\in\N$. Since $G$ is nilpotent, 
$G_n=\{e\}$ for some $n$ in $\N$. Also, each $G_m$ is connected and characteristic in $G$, $G_m\subset G_{m-1}\subset G^0$,  and 
$G_{m-1}/G_m$ is central in $G/G_m$, $m\in\N$. Let $\varphi'_m:=\varphi|_{G_m}$ and $\bar\varphi'_m\in\Aut(G_m/G_{m+1})$ be 
the automorphism induced by $\varphi'_m$, $m\in\N\cup\{0\}$. Let $\bar\varphi_m\in\Aut(G/G_m)$ be the automorphism of $G/G_m$ 
induced by $\varphi$, $m\in\N$. By \Cref{multi}, we get that $R(\varphi'_m)=1$ and $R(\bar\varphi'_m)=1$ for each $m\in\N\cup\{0\}$.  
As $R(\bar\varphi)<\infty$, where $\bar\varphi\in\Aut(G/G^0)$ and $G^0/G_1$ is central in $G/G_1$, by \Cref{inv}\,(5), 
$R(\bar\varphi_1)<\infty$. Then as $G_{m-1}/G_m$ is central in $G/G_m$, applying \Cref{inv}\,(5) successively 
to $G/G_m$ for $m\ge 2$, we get that $R(\bar\varphi_m)<\infty$, $m\ge 2$. As $G_n=\{e\}$, we get that $\bar\varphi_n=\varphi$, 
and hence $R(\varphi)<\infty$. 
\end{proof}

The following is an easy consequence of \Cref{Cptgenilp} as any closed subgroup of a compactly generated nilpotent group 
is also compactly generated.

\begin{cor} Let $G$ be a compactly generated nilpotent Lie group. Let $\varphi\in\Aut(G)$ and let $H$ be a closed normal 
$\varphi$-invariant subgroup of $G$. Let $\varphi'=\varphi|_H$ and let $\bar\varphi$ be the automorphism of $G/H$ induced 
by $\varphi$. Then $R(\varphi)<\infty$ if and only if $R(\varphi')<\infty$ and $R(\bar\varphi)<\infty$.
\end{cor}

Now we generalise a part of \Cref{Cptgenilp} and a part of \Cref{Solv} to compactly generated solvable Lie groups. 
The following is known for polycyclic groups (see e.g.\ \cite[Theorem 2]{W}).

\begin{cor} \label{Cptgensolv}
 Let $G$ be a compactly generated solvable Lie group and $\varphi \in  \Aut(G)$. 
 Let $\Fix(\varphi)$ denote the set of fixed points of $\varphi$. Then the following hold: If $|\Fix(\varphi)|=\infty$, then $R(\varphi)=\infty$. 
\end{cor}

\begin{proof} Let $G$ and $\varphi$ be as in the hypothesis and suppose $|\Fix(\varphi)|=\infty$. We prove that $R(\varphi)=\infty$ 
by induction on the length $n(G)$ of the derived series in $G$, i.e.\ for $G_1:=\ol{[G,G]}$ and $G_{m+1}:=\ol{[G_m,G_m]}$, $m\in\N$, 
there exists $n=n(G)$ such that $G_n=\{e\}$, as $G$ is solvable. First suppose $G$ is abelian, i.e.\ $n(G)=1$. Then $R(\varphi)=\infty$ 
by \Cref{Cptgenilp}. Now suppose the assertion holds for all $G$ for which $n(G)\leq k$ for some $k\in\N$. Now suppose $G$ is such that 
$n(G)=k+1$. Then $G_1$ is a compactly generated solvable Lie group and it is a characteristic subgroup of $G$. Moreover, $n(G_1)=k$ 
and $G/G_1$ is a compactly generated abelian Lie group. Let $\varphi':=\varphi|_{G_1}$ and $\bar\varphi$ be the automorphism 
of $G/G_1$ induced by $\varphi$. If $R(\bar\varphi)=\infty$, then $R(\varphi)=\infty$. Now suppose $R(\bar\varphi)<\infty$. Then by \Cref{Cptgenilp}, 
we have that $|\Fix(\bar\varphi)|<\infty$ and hence $\Fix(\bar\varphi)$ is a finite subgroup of $G/G_1$. Since $|\Fix(\varphi)|=\infty$, we get that 
$|\Fix(\varphi')|=\infty$. As $n(G_1)=k$, by the induction hypothesis, we have that $R(\varphi')=\infty$. Now by \Cref{inv}\,(1), we get that $R(\varphi)=\infty$.
\end{proof}

The following example shows that the converse of \Cref{Cptgensolv} does not hold.

\begin{example} \label{counter} Let $G=\Z_2\ltimes\R$, where the action of $\Z_2=\{0,1\}$ on $\R$ is defined uniquely by 
$(1,0)(0,r)(1,0)=(0, -r)$ for $r\in\R$. Here, $G$ is a compactly generated solvable Lie group. Consider $i_x$, the inner 
automorphism by $x=(1,0)\in\Z_2$. It is easy to see that $\Fix(i_x)=\{(0,0), (1,0)\}=\Z_2$ and $|\Fix(i_x)|=2$. But $R(i_x)=\infty$ 
by \Cref{inner}. One can also take a torsion-free compactly generated solvable Lie group $H=\Z\ltimes\R$, where the  
action of \,$\Z$ on $\R$ factors through the action of $\Z_2$ on $\R$. It is easy to see that for $\phi=-\Id$ on $H$, $\Fix(\phi)=\{(0,0)\}$ 
while $\Fix(i_y\circ\phi)=\R$ for $y=(1,0)\in\Z$, and $R(\phi)=R(i_y\circ\phi)=\infty$ by \Cref{Cptgensolv}.

Note that $G$ does not have (topological) $R_\infty$-property as for $\psi\in\Aut(G)$ defined as $\psi(z,r)=(z,\alpha\,r)$, 
$z\in\Z_2$, $r\in\R$ where $\alpha\in\R^*$ such that $\alpha\ne \pm\,1$, $R(\psi)=2$. Moreover, $H$ as above also does not have 
this property as for $\varphi\in\Aut(H)$, $\varphi(z,r)=(-z,\alpha r)$, $z\in\Z$, $r\in\R$, where $\alpha$ is as above, $R(\varphi)=2$.
\end{example}

There are two different kinds of Lie groups of the form $\Z\ltimes\R$, one of them, where the $\Z$-action on $\R$ factors through 
the action of $\Z_2$ on $\R$ is discussed in \Cref{counter} above. Now we discuss the other kind.

\begin{example} 
Consider a Lie group $L:=\Z\ltimes\R$, where the conjugation action of $\Z$ on $\R$ is given by 
$(z,0)(0,r)(-z, 0)=(0,\alpha^z r)$, $z\in\Z$, $r\in\R$, for some $\alpha\in\R^*\mi\{\pm\,1\}$. (All these groups are isomorphic for different $\alpha$, 
so we can fix $\alpha$.) Note that $L^0=\R$ and it is characteristic in $L$. Let $\varphi\in\Aut(L)$. Let $\varphi':=\varphi|_\R$ and let $\bar\varphi$ 
be the automorphism induced by $\varphi$ on $L/L^0$, which is isomorphic to $\Z$. Then $\varphi'=\beta\,\Id$ for some $\beta\in\R^*$, and 
$\bar\varphi=\pm\,\Id$. Now we show that $\bar\varphi=\Id$. If possible, suppose $\varphi(z)=-zx_z$ for some $x_z\in\R$ which depends 
on $z$. As $\R$ is abelian and $\varphi$ is a homomorphism, for $z\in\Z$ and $r\in\R$, $\varphi(zrz^{-1})=\alpha^z\beta r=\alpha^{-z}\beta r$. 
Hence, $\alpha^{2z}=1$ which leads to a contradiction as $\alpha\ne\pm\,1$ and $z$ can be chosen to be nonzero. Thus $\bar\varphi=\Id$ 
and $R(\bar\varphi)=\infty$. Now by \Cref{endo}, $R(\varphi)=\infty$. Since this holds for all $\varphi\in\Aut(L)$, we get that $L$ has 
topological $R_\infty$-property.
\end{example}
 
\section{Twisted conjugacy classes in connected Lie groups}
 
In this section we prove that any nontrivial connected Lie group $G$ has infinitely many conjugacy classes, and if $G$ is not nilpotent, then 
$\Aut(G)$ admits a subgroup of finite index in which every automorphism has infinite Reidemeister number.

For a connected Lie group $G$, if $\varphi=i_g$ is the inner automorphism by $g\in G$, then $R(\varphi)=R(\Id)$ is the 
same as the number of conjugacy classes in $G$. If $G$ is solvable, then $G/\ol{[G,G]}$ is nontrivial, abelian, connected and infinite, 
it follows that $R(i_g)=R(\Id)=\infty$. We will now show that this also holds for any connected Lie group $G$. 
 
\begin{thm}{\label{ConnLie}}
Let $G$ be a nontrivial connected Lie group. Then the number of conjugacy classes in $G$ is infinite. Further, if $G$ is not nilpotent, 
 then there exists a subgroup $L$ of finite index in $\Aut(G)$, such that $R(\alpha)=\infty$ for every $\alpha\in L$; in particular, there exists 
 $n \in \N$ such that $R(\varphi^n) = \infty$ for every $\varphi\in\Aut(G)$.
 \end{thm}
 
 We first prove the following useful result which generalises Lemma 3.2 of \cite{CS}. 
  
 \begin{lem} \label{Autgn} Let $G$ be a connected Lie group which is not nilpotent. Then $\Aut(G)$ has a normal subgroup $L$ of finite 
 index such that $1$ is an eigenvalue of $\du\alpha$ for every $\alpha\in L$. In particular, for $n=|\Aut(G)/L|$, $1$ is an eigenvalue of 
 \,$\du\varphi^n$ for every $\varphi\in\Aut(G)$.
 \end{lem}
 
 \begin{proof} Let $K$ be the maximal compact connected central subgroup of $G$. Then $K$ is characteristic in $G$ and 
 $G/K$ has no nontrivial compact connected central subgroups. Hence $\Aut(G/K)$ is almost algebraic, i.e.\ it is a closed subgroup of finite 
 index in an algebraic group  (see \cite{D1}, see also \cite{DS}). In particular, $\Aut(G/K)$ has finitely many connected components. Let 
 $\Aut(G/K)^0$ denote the connected component of the identity in $\Aut(G/K)$.  As the natural map $\varphi\mapsto\bar\varphi$ from 
 $\Aut(G)$ to $\Aut(G/K)$ is a homomorphism, we have that the subgroup $L:=\{\psi\in\Aut(G)\mid\bar\psi\in\Aut(G/K)^0\}$ is a normal 
 subgroup of finite index in $\Aut(G)$. Let $n=|\Aut(G)/L|$. Then $\varphi^n\in L$ for every $\varphi\in\Aut(G)$.

 If $1$ is an eigenvalue of $\du\alpha$ for every $\alpha\in \Aut(G/K)^0$, then $1$ is an eigenvalue of $\du\alpha$ for every $\alpha\in L$. 
 Moreover, $G/K$ is not nilpotent as $G$ is not nilpotent. Hence we may replace $G$ by $G/K$ and assume that $\Aut(G)^0$ is a 
 subgroup of finite index in $\Aut(G)$. Now we show that $1$ is an eigenvalue of $\du\alpha$ for every $\alpha\in\Aut(G)^0$. 
 
 Let $G=SR$ be a Levi decomposition, where either $G$ is solvable or $S$ is nontrivial. Suppose $G$ is solvable. For the nilradical $N$ 
 of $G$, we show that $\Aut(G)^0$ acts trivially on $G/N$. Note that $\Aut(G)^0$ is a connected Lie group and $M:=\Aut(G)^0\ltimes G$ is 
 also a connected Lie group. Let $\Aut(G)^0=S'R'$ be a Levi decomposition. Then $R'G$ is the radical of $M$. As $[R'G,G]$ is connected, 
 nilpotent and normal in $M$ and it is contained in $G$, we get that $[R'G, G]\subset N$ and hence $R'$ acts trivially on $G/N$. Note that  
 $[S',G]$ is connected and it is contained in $G$; moreover, it is contained in the nilradical (say) $N'$ of $M$, and hence it is nilpotent. 
 Thus $[S',G]\subset (N'\cap G)^0\subset N$. Hence $S'$ also acts trivially on $G/N$. Hence, $\Aut(G)^0$ acts trivially on $G/N$. Thus 
 every element $\alpha\in\Aut(G)^0$ acts trivially on $G/N$. Thus the assertion follows when $G$ is solvable (but not nilpotent).
 
 Now suppose $G$ is not solvable. Then $G/R$ is semisimple, and $\Aut(G)^0$ acts on $G/R$ by
 inner automorphisms. Hence we may replace $G$ by $G/R$ and assume that $G$ is semisimple and $\Aut(G)^0$ 
 consists of inner automorphisms. Let $\alpha\in\Aut(G)^0$. Then $\alpha=i_x$ for some $x$ in $G$. The center $Z$ of $G$ is discrete.  
 We may replace $G$ by $G/Z$ and assume that $G$ is a connected semisimple linear group with trivial center. Now $G$ is almost algebraic. If $x=e$, 
 then $\alpha=\Id$, and the assertion follows trivially. Now suppose  $x\ne e$. If $x$ has finite order, i.e.\ $x^n=e$ for some $n\in\N\mi\{1\}$, 
 then $x$ belongs to a maximal torus (maximal compact connected abelian subgroup) $T$ of $G$, and $\alpha$ acts trivially on $T$, and hence 
 $1$ is an eigenvalue of $\du\alpha$. Now suppose $x$ has infinite order. As $G$ is almost algebraic, there exists an almost algebraic subgroup 
 $A$ of $G$ containing $x$. Then $A$ is abelian and has finitely many connected components. Moreover, $A^0$ is nontrivial as $x$ does not have 
 finite order. As $x$ centralises $A^0$, $\alpha$ acts trivially on $A^0$. Hence $1$ is an eigenvalue of $\du\alpha$. Thus the assertion follows.
 \end{proof}
 
 \begin{remark} \label{autg} The proof of Lemma {\rm\ref{Autgn}} shows that for any connected Lie group $G$, $\Aut(G)$ has a 
 subgroup of finite index 
 which acts trivially on $R/N$ and it acts by inner automorphisms on $G/R$, where $R$ and $N$ are respectively the 
 radical and the nilradical of $G$. Note also that any subgroup of finite index in $\Aut(G)$ contains $\Aut(G)^0$.
 \end{remark}
 
 \begin{proof}[Proof of Theorem {\rm\ref{ConnLie}}] Let $G$ be a connected Lie group. We first show that $R(\Id)=\infty$. If $G$ is solvable, the
 assertion follows from \Cref{Solv}; it is also easy to show this directly as $G/\ol{[G,G]}$ is infinite and abelian, and all its 
 conjugacy classes are singleton sets. Now suppose $G$ is not solvable. Then $G/R$ is semisimple, where $R$ is the radical of $G$, and 
by \Cref{endo}, we may assume that $G$ is semisimple. 
 
 Suppose $G$ is compact. Then $G$ has a maximal compact connected abelian subgroup (maximal torus) $T$ which is nontrivial. Let 
 $F_n$ be the set of all $n$th roots of unity in $T$, $n\in\N$. Then $|F_n|\ge n$, $F_n\subset F_{mn}$, $m,n\in\N$, and for every $x\in F_n$, 
 $\{gxg^{-1}\mid g\in G\}$ consists of $n$th roots of unity. Then for primes $p, q$, $p\ne q$, $x\in F_p\mi\{e\}$ and $y\in F_q$,  
 $[x]_{\Id}\cap [y]_{\Id}=\emptyset$. Since there are infinitely many primes, we have that $G$ has infinitely many conjugacy classes. Note that for 
 any connected Lie group $G$, if it has a nontrivial maximal torus, then it has infinitely many conjugacy classes. However, there are connected 
semisimple Lie groups without any nontrivial compact connected abelian subgroups; for example $\widetilde{\SL(2,\R)}$, the universal cover 
of $\SL(2,\R)$.
 
 Now suppose $G$ is not  compact. Note that $\Ad x=\du (i_x)$, $x\in G$, and $\Ad$ is a homomorphism from $G\to\GL(\G)$, and $\Ad(G)$ 
 is semisimple, as $G$ is so. Choose $x\in G$, such that at least one eigenvalue of $\Ad x$ has absolute value other than $1$. 
 (One can choose $x$ 
 such that $\Ad x$ is a nontrivial element of $A$, where $A$ is such that $\Ad G=KA\,U$ is an Iwasawa decomposition and $A$ consists of 
 semisimple elements.) Let $\lambda$ be an eigenvalue of $\Ad x$ whose absolute value is either less than $1$ or greater than $1$.  Now
  $\lambda^n$ is an eigenvalue of $\Ad x^n$, and $|\lambda^n|\ne 1$, $n\in\N$. Note that $|\lambda^n|=|\lambda|^n$ converges to zero or infinity. 
  Since the eigenvalues of $\Ad (gxg^{-1})=\Ad g\,\Ad x\,\Ad g^{-1}$ are same as that of $\Ad x$, $[x^m]_{\Id}\ne [x^n]_{\Id}$, for all $m,n\in\N$ 
  such that $m\ne n$. Thus $G$ has infinitely many conjugacy classes.

To prove the second assertion, suppose that $G$ is not nilpotent. By \Cref{Autgn}, there exists a  normal subgroup (say) $L$ of finite index 
in $\Aut(G)$ such that $1$ is an eigenvalue of $\du\alpha$ for every $\alpha\in L$. Moreover, $L$ acts on $G/R$ by inner automorphisms and 
it acts trivially on $R/N$ (see Remark \ref{autg}). If $G=R$, then by \Cref{Solv}, $R(\alpha)=\infty$ for every $\alpha\in L$. Now suppose $G$ 
is not solvable. Then $G/R$ is semisimple. Then $L$ acts on $G/R$ by inner automorphisms and hence $R(\bar\alpha)=\infty$ for every 
$\alpha\in L$, where $\bar\alpha$ is the automorphism of $G/R$ induced by $\alpha$. Thus $R(\alpha)=\infty$ for every $\alpha\in L$, 
and the second assertion holds. The last assertion follows easily from the second assertion for $n=|\Aut(G)/L|$. 
 \end{proof}

\section{Lie groups with topological $R_\infty$-property}

 In this section, we first discuss sufficient conditions for connected solvable Lie groups to have topological $R_\infty$-property in \Cref{Solv-infty}. Using
  the theorem,  we prove that for $n\geq 2,$ the group of invertible $n\times n$-upper triangular real matrices and its quotient group modulo its center have 
  topological $R_\infty$-property (see \Cref{upper} and \Cref{pbn}). We show that $\SO(2,\R)\ltimes\R^2$ and certain groups of the form 
  $\R\ltimes\R^2$ do not have this property (see \Cref{non-ex}), while the Walnut group has this property (see \Cref{walnut}). 
  We also show that $\SL(2,\R)$ and $\GL(2,\R)$ have this property (see \Cref{SL2}). We give examples of many solvable and non-solvable 
  Lie groups with this property. 
 
 We first note some properties of Cartan subgroups in a Lie group which will be useful in proving \Cref{Solv-infty}. 
 We know that a Cartan subgroup of a connected solvable Lie group $G$ is 
 a maximal nilpotent subgroup $C$ in $G$ such that $G=CN$, where $N$ is the nilradical of $G$ \cite[Corollary 5.2]{MS2}; this can also be deduced from 
 Proposition 3.1 in \cite{MS1} and Lemma 9 in \cite{Wi}.  All the Cartan subgroups of a solvable Lie group are connected \cite[Theorem 1.9]{Wu}, and they 
 are conjugate to each other \cite[ Proposition 6]{Wi}.  If $C$ is a Cartan subgroup of $G$, then so is $\varphi(C)$ for every $\varphi\in\Aut(G)$, and as 
 $G$ is solvable, $\varphi(C)=x^{-1}Cx$ for some $x\in G$, i.e.\ $i_x\circ\varphi$ keeps $C$ invariant. Thus we have that a connected solvable Lie group 
 $G$ has topological $R_\infty$-property if and only if $R(\varphi)=\infty$ for every $\varphi\in\Aut(G)$ which keeps a Cartan subgroup invariant. 
 
We already know that for any automorphism $\varphi$ of a connected solvable Lie group $G$, the corresponding automorphism $\bar\varphi_1$ 
 on $G/N$ has finite order, and hence all the eigenvalues of $\du\bar\varphi_1$ have absolute value $1$ (see Remark 3.3 in \cite{CS}, see also \Cref{autg}). 
So if $\du\bar\varphi_1$ has a real eigenvalue (this holds in particular if the dimension of $G/N$ is odd), then it must be either $1$ or $-1$. 
\Cref{Solv-infty} shows that under a certain condition, $1$ is an eigenvalue of $\du\bar\varphi_1$ and hence that of $\du\varphi$, and  
$R(\varphi)=\infty$. The theorem will enable us to  prove topological $R_\infty$-property for many solvable Lie groups.

\begin{thm} \label{Solv-infty} Let $G$ be a connected solvable Lie group which is not nilpotent. Suppose $G$ has a closed connected 
one-dimensional normal subgroup $V$ which is not central in $G$.  Then the following hold:
\begin{enumerate}
\item[$(1)$]  If $V$ is $\varphi$-invariant, for some $\varphi\in\Aut(G)$, then $1$ is an eigenvalue of $\du\varphi$ as well as that of $\du\bar\varphi$ and 
$R(\varphi)=R(\bar\varphi)=\infty$, where $\bar\varphi$ is the automorphism of $G/Z$ induced by $\varphi$ for the center $Z$ of $G$. 
\item[$(2)$] In particular, if $V$ as above is characteristic in $G$, then $G$ has topological $R_\infty$-property. 
\end{enumerate}
\end{thm}

\begin{proof} Let $G$ and $V$ be as in the hypothesis. Note that any compact connected abelian normal subgroup of $G$ is central in $G$. 
Since $V$ is one-dimensional and it is not central in $G$, it is a vector group. Let $N$ be the nilradical of $G$.  
Then $V\subset N$ and it is normal in $N$. Moreover, since $N$ is nilpotent and $V$ is connected, $[N,V]$ is a proper connected subgroup of $V$. 
Since $V$ is one-dimensional, it follows that $[N,V]$ is trivial, and hence  $V$ is central in $N$.

Let $C$ be a Cartan subgroup of $G$. Then $G=CN$ and $C\ne N$ and $C\ne G$ as $N$ is a proper subgroup of $G$. Since $V$ is connected, abelian 
and normal in $G$, we have that $C\cap V$ is connected \cite[Theorem 1.9]{Wu}. We show that $V\cap C=\{e\}$. If possible, suppose $V\cap C$ 
is nontrivial. Since it is connected and $V$ is one-dimensional, we get that $V\subset C$, and $V$ is a normal subgroup of $C$. Since $C$ is nilpotent, arguing 
as above, we get that $V$ is central in $C$.  As $V$ is central in $N$ and $G=CN$, we get that $V$ is central in  $G$, 
which leads to a contradiction. Thus $V\cap C=\{e\}$.

\smallskip
\noindent{\bf (1):}   
 Let $\varphi\in\Aut(G)$ be such that $\varphi(V)=V$. There exists $x\in N$ such that $\varphi (C)=x^{-1}Cx$, and hence 
 $i_x\circ\varphi(C)=C$. Since $V$ is $\varphi$-invariant and normal in $G$, we have that $i_x\circ\varphi(V)=V$. 
Moreover $(i_x\circ\varphi)|_Z=\varphi|_Z$, and hence the automorphism induced by $i_x\circ\varphi$ on $G/Z$ is 
$i_{\bar x}\circ\bar\varphi$, where $\bar x=xZ\in G/Z$. As $R(\varphi)=R(i_x\circ\varphi)$ and $R(\bar\varphi)=R(i_{\bar x}\circ\bar\varphi)$ 
(cf.\  \Cref{inner}), we may replace $\varphi$ by $i_x\circ\varphi$ and assume that $\varphi(C)=C$. Since $C$ normalises $V$, we have that 
$C$ acts on $V$ by inner automorphisms. Let $Z_C(V)$ be the centraliser of $V$ in $C$. Then $Z_C(V)$ is a closed normal subgroup of $C$. 
Note that $C$ does not centralise $V$; otherwise $CV$ would be nilpotent contradicting the fact that $C$ is  maximal nilpotent subgroup in $G$. Thus 
$Z_C(V)$ is a proper subgroup of $C$. By Theorem 1.5 of \cite{MS1}, $CV/V$ is a Cartan subgroup of $G/V$, 
and hence $CV$ is closed, and it is a connected (Lie) subgroup of $G$. Moreover, as both $V$ and $C$ are $\varphi$-invariant, so are $CV$ and $Z_C(V)$,  
and $C$ is a Cartan subgroup of $CV$ \cite[Lemma 3.12]{MS1}. Note that $Z_C(V)$ is normal in $C$, and hence in $CV$. Then $CV/Z_C(V)=L\ltimes V$, 
where $L:=C/Z_C(V)$. As $V$ is a one-dimensional vector group, and $C$ is connected, the action of $C$ on $V$ is via the action of $\R^*_+$. 
Thus $L=C/Z_C(V)$ is isomorphic to $\R^*_+$.  Let $H:=L\ltimes V$. Then $H$ is a 
connected solvable Lie group which is not nilpotent as $L$ is a Cartan subgroup of $H$ \cite[Theorem 1.5]{MS1}, and $V$ is the nilradical of $H$. 
Moreover, $\varphi$ induces an automorphism (say) $\psi$ on $H$ such that $\psi|_L$ is the automorphism induced by $\varphi|_C$ on its quotient group 
$L$ and $\psi|_V=\varphi|_V$.  As the dimension of $L$ is $1$,  $\du(\psi|_L)$ has a unique real eigenvalue. Since $H/V$ is isomorphic to $L$, by 
Remark 3.3 of \cite{CS} (see also \Cref{autg}), we have that the eigenvalue of $\du(\psi|_L)$ is a root of unity, and hence it is either $1$ or $-1$, 
i.e.\ $\psi(t)=t$ for all $t\in L$, or $\psi(t)=t^{-1}$ for $t\in L$. Also, as $V$ is a vector group, $\psi(r)=\beta r$, $r\in V$, for some $\beta\in\R^*$.

We now show that $\psi|_L=\Id$. If possible, suppose $\psi(t)=t^{-1}$ for all $t\in L$. 
Note that the conjugation action of $L$ on $V$ is given by $i_t(r)=t^\alpha r$ and $i_{t^{-1}}(r)=t^{-\alpha}r$, $r\in V$ for some $\alpha\in\R$.
 Here, $\alpha\ne 0$ as $H$ is not abelian. Then we get that $i_{t^{-1}}\circ \psi=\psi\circ i_t$ for all $t\in\R$. This implies that $t^{-\alpha}\psi(r)=\psi(t^{\alpha}r)$ and hence 
  $t^{-\alpha}\beta r=t^{\alpha}\beta r$ for all $r\in V$. As $\beta\ne 0$, we get that $t^{2\alpha}=1$ which would imply that $\alpha=0$, this leads to a contradiction. 
  Thus $\psi|_L=\Id$. Since $\psi|_L$ is the automorphism induced by 
$\varphi|_C$, we get that $1$ is an eigenvalue of $\du(\varphi|_C)$, and hence that of $\du\varphi$. By \Cref{Solv}, $R(\varphi)=\infty$. Also since the center $Z$ is 
contained in $C$, and hence in $Z_C(V)$, we get that $1$ is an eigenvalue of $\du(\bar\varphi|_{(C/Z)})$. Hence $1$ is an eigenvalue of $\du\bar\varphi$. 
Thus $R(\bar\varphi)=\infty$. (We may note here that $L$ is isomorphic to $G/(Z_C(V)N)$, and hence the eigenvalue $1$ of $\du(\psi|_L)$ is an 
eigenvalue of the quotient Lie algebra automorphism $\du\bar\varphi_1$, where $\bar\varphi_1\in\Aut(G/N)$ is induced by $\varphi$.)

\smallskip
\noindent{\bf (2):} Since $V$ is characteristic in $G$, from (1) we have that  $R(\varphi)=\infty$ for every $\varphi\in\Aut(G)$. Hence $G$ has 
topological $R_\infty$-property. 
\end{proof}

In the following Corollary, $\R^*=\R\mi\{0\}$ is the usual multiplicative group, its connected component of the identity $\R^*_+$ is isomorphic 
to $\R$, and they both act on the additive group $\R$ by automorphisms. Note that both $\R^*$ and $\R^*_+$ do not have (topological) 
$R_\infty$-property, as the automorphism $x\mapsto x^{-1}$ on $\R^*$ (resp.\ $\R^*_+$) has Reidemeister number $2$ (resp.\ $1$).

\begin{cor} \label{axpb}
$\R^*\ltimes\R$ and $\R\ltimes\R$ have topological $R_\infty$-property. 
\end{cor}

\begin{proof} Both the groups $\R^*\ltimes\R$ and $\R\ltimes\R$ are solvable but not nilpotent. Moreover, 
$\R^*\ltimes\R$ is isomorphic to $\left\{\begin{bmatrix} 
a & b\cr
0& 1\cr\end{bmatrix}
\mid a,b\in \R, a\ne 0\right\}$, and $\R\ltimes\R$ is isomorphic to 
$\left\{\begin{bmatrix} 
a & b\cr
0& 1\cr
\end{bmatrix}
\mid a,b\in \R, a>0 \right\}$. 
Under this identification, $\R\ltimes\R$ is a closed connected characteristic subgroup of index $2$ in $\R^*\ltimes\R$. Hence by \Cref{inv}\,(1), 
it is enough to prove that $G:=\R\ltimes\R$ has topological $R_\infty$-property. The nilradical $N$ of $G$ is isomorphic to $\R$ which is a characteristic 
subgroup of $G$, and it is not central in $G$. Then by \Cref{Solv-infty}\,(2), $G$ has topological $R_\infty$-property. 
\end{proof}

Let $U_n$ denote the group of $n\times n$ strictly upper triangular real matrices (with all the diagonal entries equal to $1$).
It is shown in \cite[Theorem 4] {N2} that $U_n$ does not have (topological) $R_\infty$-property if $n\ge 3$; this is also true for $n=2$. Here, 
we consider groups of invertible $n\times n$ upper triangular real matrices (with nonzero diagonal entries), $n\ge 2$. Note that for $n=1$, such 
a group is isomorphic to the multiplicative abelian group $\R^*$ which does not have (topological) $R_\infty$-property, as the Reidemeister 
number of the automorphism $x\mapsto x^{-1}$ on $\R^*$ is $2$.

\begin{thm} \label{upper} The group of invertible $n\times n$ upper triangular real matrices has topological $R_\infty$-property for every $n\ge 2$.
\end{thm}

\begin{proof}
Let $G$ be the group of invertible $n\times n$ upper triangular real matrices for some $n\in\N$, $n\ne 1$. Then it is a solvable Lie group which is 
not nilpotent. Also, $G^0$, the connected component of the identity in $G$, is a group of 
elements in $G$ with positive entries on the diagonal and it is a characteristic subgroup of finite index in $G$. Now by \Cref{inv}\,(1), it
is enough to show that $G^0$ has topological $R_\infty$-property. Hence we may replace $G$ by $G^0$, and assume that $G$ is 
a connected solvable Lie group consisting of upper triangular real matrices with positive entries on the diagonal, i.e.\ 
$$
G=\{(a_{ij})\mid a_{ij}\in\R, a_{ij}=0 \mbox{ if } j<i \mbox{ and } a_{ii}> 0, 1\le i,j\le n\}.$$ 
The nilradical $N$ of $G$ is the group $ZU$, where $Z$ is the center of $G$ consisting of diagonal matrices  with same diagonal entries 
in $G$ (positive scalar matrices), and $U:=U_n$ is as above. Since $G$ is solvable, the last nontrivial subgroup in the derived series of $G$ is 
$$
V=\{(a_{ij})\mid a_{1n}\in\R, a_{ii}= 1, a_{ij}=0\mbox{ if }1\le i\ne j\le n \mbox{ except for }a_{1n}\}.$$

Note that $V$ is closed, connected and characteristic in $G$ and it is the center of $U$. But $V$ is not central in $G$; for if $Z_G(V)$ 
denotes the centraliser of $V$ in $G$, then
$$Z_G(V)=\{(a_{ij})\in G\mid a_{11}=a_{nn}\}\ne G.$$ 
As the dimension of $V$ is $1$, by  \Cref{Solv-infty}\,(2), $G$ has topological $R_\infty$-property.  
\end{proof}

Note that in \cite{BB2}, Borel subgroups of a semisimple algebraic group over an algebraically closed field of characteristic zero are considered and
it is shown that they have algebraic $R_\infty$-property, i.e.\ $R(\phi)=\infty$ for every algebraic automorphism $\phi$. Borel subgroup in a connected 
Lie (resp.\ algebraic) group is a maximal closed (Zariski) connected solvable subgroup. In \Cref{upper}, we have seen that in $\GL(n,\R)$, $n\geq 2$, 
a Borel subgroup has topological $R_\infty$-property, and this holds for all Borel subgroups as they are conjugate to each other. 
In \cite{LR2}, the quotient group of a Borel subgroup modulo its center is considered over some specific integral domains which do not include $\R$. 
Here we show the following.

\begin{prop} \label{pbn} Let $G$ be a group of invertible $n\times n$ upper triangular real matrices for $n\ge 2$ and let $Z$ be the center of $G$. 
Then $G/Z$ has topological $R_\infty$-property. 
\end{prop}

\begin{proof} The center $Z$ of $G$ is also the center of $\GL(n,\R)$, which is the group of diagonal matrices with same entries, i.e.\ scalar matrices. 
Here, both $G$ and  $G/Z$ are solvable Lie groups with finitely many connected components and $G^0Z$ is closed in $G$. Note that $G^0Z/Z$ is the 
connected component of the identity in $G/Z$ and it is a characteristic subgroup of $G/Z$. Moreover, $G^0Z/Z$ is isomorphic to $G^0/(G^0\cap Z)$.  
By \Cref{inv}\,(1), it is enough to show that $G^0/(G^0\cap Z)$ has topological $R_\infty$-property.  As $G^0\cap Z$ is the center of $G^0$, we may replace 
$G$ by $G^0$ and assume that $G$ is a connected Lie group of upper triangular real matrices with positive diagonal entries, and $Z$ is the group 
of diagonal matrices whose all entries are same and positive. Let $V$ be as in the proof of \Cref{upper}. Then $V'=VZ/Z$ is a closed connected 
characteristic vector subgroup of dimension $1$ in $G/Z$. It is easy to see that $V'$ is not central in $G/Z$. Now by \Cref{Solv-infty}\,(2), 
$G/Z$ has topological $R_\infty$-property. 
\end{proof}

Note that \Cref{pbn} also holds if $G$ is any Borel subgroup of $\GL(n,\R)$, $n\geq 2$. 
Observe that if $H$ is the subgroup consisting of upper triangular matrices in $\SL(n,\R)$,  then $H/H^0$ is finite and $H^0$ is isomorphic 
to $G^0/(G^0\cap Z)$, where $G$ and $Z$ are as in \Cref{pbn}. As shown in the proof of \Cref{pbn}, $H^0$ has  topological 
$R_\infty$-property, and so does $H$, by Lemma \ref{inv}\,(1), Therefore, all Borel subgroups of $\SL(n,\R)$ have this property, $n\geq 2$. 
We now construct some more connected $m$-step solvable non-nilpotent Lie groups with this property for all $m\geq 2$.

\begin{example}
For the unipotent group $U_n$ of $n\times n$ strictly upper triangular real matrices,  let $H_n=\R\ltimes U_n$, $n\geq 2$, where the action 
of any $t\in\R^*_+$ on $U_n$ is given by $(a_{ij})\mapsto (b_{ij})\in U_n$, where 
$b_{ij}=a_{ij}=0$, if $j< i$, $b_{ii}=a_{ii}=1$ and $b_{ij}=t^ka_{ij}$, where $k=j-i$ if $j>i$, for $i,j\in\{1,\ldots,n\}$. Then each $H_n$ is a 
connected solvable Lie group and satisfies the conditions stated in \Cref{Solv-infty}\,(2), as the last group in the central series of $U_n$ 
is characteristic in $G$ and is not central in $G$. Hence $H_n$ has topological $R_\infty$-property, 
for every $n\in\N$, $n\geq 2$. (The case $n=2$, where $H_2$ is isomorphic to $\R\ltimes\R$, is already covered in \Cref{axpb}.)
\end{example}

The condition in \Cref{Solv-infty} is necessary, as there are examples of both simply connected and non simply connected $2$-step solvable Lie groups without any 
characteristic one-dimensional non-central subgroups, which do not have (topological) $R_\infty$-property. We first note that $\SO(2,\R)$, 
the special orthogonal group (the group of rotations on $\R^2$), is a one-dimensional torus 
and the normaliser of $\SO(2,\R)$ in $\GL(2,\R)$ is $(\Hk\cup\{I\})\D\,\SO(2,\R)$, where $I$ is the identity matrix, $\Hk=\{\varrho,\xi\}$ with  
$$
\varrho=\begin{bmatrix} 
1 & 0\cr
0& -1\cr
\end{bmatrix} \ \mbox{ and } \ \xi=\begin{bmatrix} 
0 & 1\cr
1& 0\cr
\end{bmatrix},$$
and $\D=\{rI\mid r\in\R^*\}$, the center of $\GL(2,\R)$. The centraliser of $\SO(2,\R)$ in $\GL(2,\R)$ is $\D\,\SO(2,\R)$.

\begin{example} \label{non-ex} Let $G=\R\ltimes\R^2$, with $\R$ being a multiplicative group isomorphic to $\R^*_+$, in which the action of  
$\R$ on $\R^2$ is given by $(x,y)\mapsto (tx, t^{-1}y)$, $(x,y)\in\R^2$, $t\in\R=\R^*_+$. 
Let $\phi\in\Aut(G)$ be defined as follows: $\phi(t, (x,y))=(t^{-1}, (\alpha y,\alpha x))$, for $(x,y)$ and $t$ as above, and any fixed   
$\alpha\in\R\mi\{0\}$; i.e.\ $\phi|_{\R^2}=\alpha\xi$ for $\xi$ as above. Now if we take any $\alpha\in \R\mi\{0,1\}$, then $\du\phi$ does not have
$1$ as an eigenvalue. Then $R(\phi)=1$ by \Cref{Solv}. Note that the center of $G$ is trivial and $G$ does not have any one-dimensional 
characteristic subgroup.

Now we consider a different action of $\R$ on $\R^2$.  Let $G=\R\ltimes \R^2$, where $\R$ is an additive group and the action 
of $\R$ on $\R^2$ is given by $x\mapsto e^{2i\pi t}x$, $x\in\R^2$, $t\in\R$. Here, the center $Z$ of $G$ is contained in $\R$, and it is 
isomorphic to $\Z$. Neither $G$ nor $G/Z$ has any one-dimensional characteristic (non-central) subgroup. 
Here $G$ is simply connected, while $G/Z$ is not. Let $\varphi:G\to G$ be defined as follows: 
$\varphi(t,x)=(-t, \varphi'(x))$, $t\in\R$, $x\in\R^2$, where $\varphi'=2\varrho$, for $\varrho$ as above. Then $\varphi\in\Aut(G)$. Since $G$ 
is solvable and $1$ is not an eigenvalue of $\varphi$, by \Cref{Solv}, $R(\varphi)=1$. Let  $\varphi_1$ be the automorphism of $G/Z$ which 
is isomorphic to $\SO(2,\R)\ltimes\R^2$. We have that $R(\varphi_1)=1$. Note also that $NZ/Z$ is the nilradical of $G$ which is isomorphic 
to $\R^2$. If $\bar\varphi$ (resp.\ $\bar\varphi_1$) is the 
automorphism induced by $\varphi$ (resp.\ $\varphi_1$) on $G/N$ (resp.\ $(G/Z)/(NZ/Z)$), then $-1$ is the only eigenvalue of both $\bar\varphi$ 
on $G/N$ and $\du\bar\varphi_1$. Thus both $\R\ltimes\R^2$ and $\SO(2,\R)\ltimes\R^2$ defined as above do not have (topological) 
$R_\infty$-property. Note that one can choose any $\varphi'$ in $\GL(2,\R)$ that normalises $\SO(2,\R)$ but does not centralise 
$\SO(2,\R)$, and $1$ is not an eigenvalue of $\varphi'$; e.g.\ $\varphi'\in\{r\varrho,r\xi\mid r\in\R\mi \{0,1\}\}$ and the corresponding 
$\varphi$ as well as $\varphi_1$ will not have $1$ as an eigenvalue, and their Reidemeister number will be $1$. 
\end{example}

Examples of connected solvable Lie groups with topological $R_\infty$-property mentioned above are all simply connected. There is 
a connected solvable Lie group which is not simply connected and has topological $R_\infty$-property, even though it does not satisfy the 
condition in \Cref{Solv-infty}\,(2). The group $W:=\SO(2,\R)\ltimes \Hm/D$ is known as the {\it Walnut} group (see, for example,  
\cite{DYS} or \cite{DS}), where $\Hm:=U_3$ is the $3$-dimensional Heisenberg group, the center $Z$ of $\Hm$ is the center of 
$\SO(2,\R)\ltimes\Hm$ and $D$ 
is an infinite discrete (central) subgroup of $Z$ (see the description of the action of $\SL(2,\R)$ on $\Hm$, which acts trivially on the center of 
$\Hm$, in Section 5 before \Cref{PSL2}). As noted above, the quotient group of $W$ modulo its center is $\SO(2,\R)\ltimes\R^2$, which 
does not have (topological) $R_\infty$-property.

\begin{prop} \label{walnut} The Walnut group has topological $R_\infty$-property. 
\end{prop}

\begin{proof}
The Walnut group $W=\SO(2,\R)\ltimes \Hm/D$ is a 4-dimensional connected solvable Lie group whose center $K:=Z/D$, is a one-dimensional 
torus. Let $\varphi\in\Aut(W)$. We show that $R(\varphi)=\infty$. By \Cref{Solv}, it is enough to show that $1$ is an eigenvalue of $\du\varphi$. 
Let $\varphi'=\varphi|_N$ and $\bar\varphi$ be the automorphism of $W/N$ induced by $\varphi$, where $N=\Hm/D$ is the nilradical of $W$. Since 
$W/N=\SO(2,\R)$ is a one-dimensional torus, and it has only two automorphisms and $\du\bar\varphi=\pm\,\Id$. If $\bar\varphi=\Id$, then 
$R(\bar\varphi)=\infty$, and hence $R(\varphi)=\infty$. Similarly, $K$ also has only two automorphisms. Thus if $\varphi|_K=\Id$, 
then $|\Fix(\varphi)|=\infty$, and  by \Cref{Solv}, $R(\varphi)=\infty$. Now suppose $\bar\varphi(x)=x^{-1}$ for all $x\in W/N=\SO(2,\R)$, and 
$\varphi(x)=x^{-1}$ for all $x\in K$. Note that  $T:=\SO(2,\R)\times K$ is a maximal torus (maximal compact connected abelian subgroup) 
in $W$ (here, $T$ is also a Cartan subgroup of $W$). Then $\varphi(T)$ is also a maximal torus in $W$ and it is conjugate to $T$ by an element 
of $N$. Thus $i_g\circ\varphi(T)=T$, for some $g\in N$. Observe that the action of $i_g\circ\varphi$ on $W/N$ and on $Z/D$ is same as that of 
$\varphi$. As $R(\varphi)=R(i_g\circ\varphi)$ by \Cref{inner}, we may replace $\varphi$ by $i_g\circ\varphi$ and assume that 
$\varphi(T)=T$. Now for $x\in \SO(2,\R)$, $\varphi(x)=x^{-1}k$ for some $k\in K$ which depends on $x$.

Note that $\Hm$ is the simply connected covering group of $N=\Hm/D$ with the covering map given by the natural projection $p:\Hm\to \Hm/D$.  
There exists $\psi\in\Aut(\Hm)$ such that $\psi$ keeps $D$ invariant and $\varphi'=\bar\psi$ is the automorphism induced by $\psi$ on $N$. 
Let $\varphi'_1$ be the automorphism  induced by $\varphi'$ on $N/K=\R^2$.  As $N/K$ and $\Hm/Z$ are isomorphic, we may assume that 
$\varphi'_1=\psi'_1$ as an element of $\GL(2,\R)$, where $\psi'_1$ is the automorphism of $\Hm/Z$ induced by $\psi$. Since $\psi|_Z=-\Id$, 
it follows that $\det\psi'_1=-1$. 
Thus $\det\varphi'_1=-1$.

Let $\varphi_1$ be automorphism of $W/K$ induced by $\varphi$. Here, $W/K$ is isomorphic to $\SO(2,\R)\ltimes \R^2$ and $\varphi_1|_{\R^2}$ 
is same as $\varphi'_1$. As $\varphi_1(x)=x^{-1}$, $x\in\SO(2,\R)$, it follows that $\varphi'_1$ normalises $\SO(2,\R)$ in $\GL(2,\R)$, but it does 
not centralise $\SO(2,\R)$ (see Example \ref{non-ex}). Moreover as $\det\varphi'_1=-1$, we get that $\varphi'_1\in \Hk\,\SO(2,\R)$, where 
$\Hk=\{\varrho,\xi\}$ is as above and $\Hk$ consists of elements of order $2$. As $\SO(2,\R)$ is compact and $\Hk$ normalises $\SO(2,\R)$, 
we get that $\varphi'_1$ is contained in a compact subgroup of $\GL(2,\R)$, and hence all its eigenvalues have absolute value $1$. This together 
with the fact that $\det \varphi'_1=-1$, implies that the eigenvalues of $\varphi'_1$ are $1$ and $-1$. Hence $1$ is an eigenvalue of $\du\varphi_1$,  
and hence that of $\du\varphi$. This implies that $R(\varphi)=\infty$. Since this holds for all $\varphi\in\Aut(W)$, we have that $W$ has 
topological $R_\infty$-property. 
\end{proof}

For connected semisimple Lie groups $G$, we have from \Cref{ConnLie} that for some $m\in\N$, $R(\varphi^m)=\infty$ for every $\varphi\in\Aut(G)$. We also know from 
\cite{N1}, that for $n\ge 3$, $\SL(n,\R)$ and $\GL(n,\R)$ have (topological) $R_\infty$-property. For $n=2$, we have the following:

\begin{thm} \label{SL2}
$\SL(2,\R)$ and $\GL(2,\R)$ have topological $R_\infty$-property.
\end{thm}

\begin{proof} {\bf Step 1:} Let $G:=\SL(2,\R)$. It is known that $\Aut(G)$ has two connected components and $\Aut(G)^0$ consists of the 
inner automorphisms of $G$. In fact, $\Aut(G)=\PGL(2,\R)$, i.e.\ the group of inner automorphisms of $\GL(2,\R)$ restricted to $\SL(2,\R)$, 
which is normal in $\GL(2,\R)$. By \Cref{ConnLie}, $R(\Id)=\infty$, and \Cref{inner} implies that for all $g\in G$, $R(i_g)=\infty$. 
Thus, if $\varphi\in \Aut(G)^0$, $R(\varphi)=\infty$.

Now let $\varphi\in\Aut(G)$ be defined as follows: $\varphi(g):=\xi\,g\,\xi^{-1}$, $g\in G$, where $\xi\in \GL(2,\R)$ is as defined before  
\Cref{non-ex}. Moreover, $\varphi\not\in \Aut(G)^0$ as $\det \xi=-1$ and $\xi\not\in\SL(2,\R)$.  For 
$$
g=\begin{bmatrix} 
a & b\cr
c& d\cr
\end{bmatrix}\in\SL(2,\R), \ \  
\varphi(g)=\xi\,g\,\xi^{-1}=\begin{bmatrix} 
0 & 1\cr
1& 0\cr
\end{bmatrix}\cdot \begin{bmatrix} 
a & b\cr
c& d\cr
\end{bmatrix}\cdot\begin{bmatrix} 
0 & 1\cr
1& 0\cr
\end{bmatrix}=\begin{bmatrix} 
d & c\cr
b& a\cr
\end{bmatrix}.$$

Now it is enough to show that $R(\varphi)=\infty$, as for any $\psi\in\Aut(G)\mi\Aut(G)^0$, $\psi=i_x\circ\varphi$ for some $x\in G$, 
and by \Cref{inner}, $R(\psi)=R(\varphi)$.

Let $G=KA\,U$ be an Iwasawa decomposition, where $K$ is a compact connected group of rotations on $\R^2$, $A$ is the group 
of diagonal matrices in $\SL(2,\R)$ with positive entries and $U$ is the group of unipotent matrices in $\SL(2,\R)$, i.e.\
$$
U=\left\{\begin{bmatrix} 
1 & r\cr
0& 1\cr
\end{bmatrix}
\mid r\in\R\right\}.
$$
Let $x_r\in U$ for some $r\in\R$, be such that $x_r=\begin{bmatrix} 
1 & r\cr
0& 1\cr
\end{bmatrix}$. Then for $g$ as above, 
$$
gx_r\varphi(g^{-1})=\begin{bmatrix} 
a^2-b^2-abr & bd-ac+adr\cr
ac-bd-bcr& d^2-c^2+cdr\cr
\end{bmatrix}.$$
As $\det g=ad-bc=1$, we get that $U\cap [x_r]_\varphi=\{x_r\}$. As $U$ is infinite, we get that $\{[x_r]_\varphi\mid x_r\in U\}$ is infinite, 
and hence $R(\varphi)=\infty$. Thus $\SL(2,\R)$ has topological $R_\infty$-property.

\medskip
\noindent{\bf Step 2:} Now suppose $G:=\GL(2,\R)$. Then $[G,G]\subset\SL(2,\R)$. Also $\SL(2,\R)$ is its own commutator group. 
Hence $\SL(2,\R)=[G,G]$ and it is a closed (normal) characteristic subgroup of $G$. Let $\varphi\in\Aut(G)$, $\varphi':=\varphi|_{[G,G]}$ 
and $\bar\varphi$ be the automorphism of $G/[G,G]$ induced by $\varphi$. Now if $R(\bar\varphi)=\infty$, then $R(\varphi)=\infty$. 
Now suppose $R(\bar\varphi)<\infty$. Note that $G/[G,G]$ is abelian, moreover, it has only two connected components, and it is 
compactly generated. By \Cref{Cptgenilp}, $|\Fix(\bar\varphi)|<\infty$. As noted above $R(\varphi')=\infty$. Now by \Cref{inv}\,(1), 
$R(\varphi)=\infty$. Since this holds for any $\varphi\in\Aut(G)$, we get that  $\GL(2,\R)$ has topological $R_\infty$-property.
\end{proof}

It seems possible to show that any abstract automorphism of $\SL(2,\R)$ is continuous, and hence it would imply that it has 
$R_\infty$-property. As our main focus is on the topological $R_\infty$-property for Lie groups, we will not discuss this here.

Consider the $3$-dimensional Heisenberg group $\Hm$. The center $Z$ of $\Hm$ is isomorphic to $\R$ and $G/Z$ is isomorphic 
to $\R^2$. There is a canonical action of $\GL(2,\R)$ on $\Hm$ which keeps the center $Z$ invariant and it acts on $G/Z$ as the 
usual action of $\GL(2,\R)$ on $\R^2$. For the sake of completeness, we define this action here. For simplicity, we denote a generic 
element $(y_{ij})$ in $\Hm$ by a tuple $(a,b,c)$ where $a=y_{12}$, $b=y_{23}$ and $c=y_{13}$. For $\psi=(x_{ij})\in \GL(2,\R)$ 
and $(a,b,c)\in\Hm$, let 
 $$
 \psi(a,b,c)= (x_{11}a+x_{12}b, x_{21}a+x_{22}b, c\det\psi-\frac{1}{2}ab\det\psi+\frac{1}{2}(x_{11}a+x_{12}b)(x_{21}a+x_{22}b)).$$
 This defines a continuous group action of $\GL(2,\R)$ on $\Hm$ by automorphisms and it keeps the center $Z$ invariant. Recall that 
 $\SL^{\pm}(n,\R)$ is the subgroup of $\GL(n,\R)$ consisting of all matrices of determinant $\pm\,1$, $n\in\N$. The action of $\SL^{\pm}(2,\R)$ 
 on $\Hm$ (which is the restriction of the action mentioned above) is such that it acts as $\pm\,\Id$ on the center $Z$ and $\SL(2,\R)$ acts 
 trivially on $Z$. Therefore, the action of $\SL^{\pm}(2,\R)$ keeps any discrete infinite subgroup $D$ of $Z$ invariant. Let 
 $\widetilde{\SL(n,\R)}$ denote the universal covering of $\SL(n,\R)$, $n\in\N\mi\{1\}$. We will consider some Lie groups below, 
 some of which are neither solvable nor semisimple. 
 
Using Theorems \ref{SL2} and \ref{Cptgenilp}, and Theorem 1 of \cite{N1} and the structure of groups, we get the following.

\begin{cor} \label{PSL2} The following groups have topological $R_\infty$-property: 
\begin{enumerate}
\item[$(1)$] $\GL(n,\R)\ltimes\R^n$, $n\geq 2$, and $\GL(2,\R)\ltimes \Hm$. 
\item[$(2)$]  $\SL(n,\R)\ltimes\R^n$, $n\geq 2$,  $\SL(2,\R)\ltimes \Hm$ and $\SL(2,\R)\ltimes \Hm/D$. 
\item[$(3)$] $\SL^{\pm}(n,\R)$ and $\SL^{\pm}(n,\R)\ltimes\R^n$, $n\geq 2$, $\SL^{\pm}(2,\R)\ltimes \Hm$ and $\SL^{\pm}(2,\R)\ltimes \Hm/D$.
\item[$(4)$] $\widetilde{\SL(2,\R)}$, $\PSL(2,\R)$ and $\PGL(2,\R)$.
\item[$(5)$] $\widetilde{\SL(n,\R)}$, $n\ge 3$, $n$ odd. 
\end{enumerate}
\end{cor}

\begin{proof} 
We know from \Cref{SL2} (resp.\ \cite{N1}) that $\SL(2,\R)$ and $\GL(2,\R)$ (resp.\ $\SL(n,\R)$ and $\GL(n,\R)$, $n\ge 3$) have 
topological $R_\infty$-property. The proper closed connected normal subgroup in each of the groups in (1) as well as (2) is the 
nilradical of that group, and hence characteristic in it. Hence by \Cref{endo}, the assertion holds for all groups in {\bf (1)} and {\bf (2)}.

\smallskip
\noindent{\bf (3):} For $n\geq 2$, the connected component of the identity in $\SL^{\pm}(n,\R)$  (resp.\ $\SL^{\pm}(n,\R)\ltimes \R^n)$ is $\SL(n,\R)$  
(resp.\  $\SL(n,\R)\ltimes \R^n$) which is a (characteristic) subgroup of index $2$ and has topological $R_\infty$-property by \Cref{SL2} 
(resp.\ by (2) above). Hence by \Cref{inv}\,(1), $\SL^{\pm}(n,\R)$ and $\SL^{\pm}(n,\R)\ltimes\R^n$ have  topological $R_\infty$-property, 
$n\geq 2$. Similarly,  $\SL^{\pm}(2,\R)\ltimes \Hm$ (resp.\ $\SL^{\pm}(2,\R)\ltimes \Hm/D$) has topological $R_\infty$-property, as it has 
$\SL(2,\R)\ltimes \Hm$ (resp.\ $\SL(2,\R)\ltimes \Hm/D)$ as a subgroup of index $2$ with the same property.

\smallskip
\noindent{\bf (4):} Let $G:=\widetilde{\SL(2,\R)}$. Then the center $Z$ of $G$ is isomorphic to $\Z$ and it is characteristic in $G$. Let 
$p:G\to\SL(2,\R)$ be the covering map. Then $\ker p$ is a subgroup of index $2$ in $Z$. Let $\varphi\in\Aut(G)$. Then $\varphi$ keeps 
$Z$ invariant. Now the only automorphisms of $\Z$ are $\pm\,\Id$. Therefore, $\varphi$ keeps any subgroup of $Z$ invariant, in particular, 
it keeps $\ker p$ invariant. Therefore, $\varphi$ induces an automorphism $\bar\varphi$ on $G/\ker p$ which is isomorphic to  $\SL(2,\R)$. 
By \Cref{SL2}, $R(\bar\varphi)=\infty$, and hence $R(\varphi)=\infty$. Thus $\widetilde{\SL(2,\R)}$ has topological $R_\infty$-property.

Note that $\PSL(2,\R)=\SL(2,\R)/\{I,-I\}$, where $I$ is the identity matrix in $\SL(2,\R)$ and $D=\{I, -I\}$ is the center of $\SL(2,\R)$. For 
$G=\widetilde{\SL(2,\R)}$ and its center $Z$ as above, $\PSL(2,\R)$ is isomorphic to $G/Z$ and $G$ is a covering group of $\PSL(2,\R)$. 
Moreover, any automorphism of $\PSL(2,\R)$ lifts to an automorphism of $G$, and any automorphism of $G$ induces an 
automorphism of $\PSL(2,\R)$ as well as that of $\SL(2,\R)$. Thus, given an automorphism $\psi$ of $\PSL(2,\R)$, there is an automorphism 
$\alpha$ of $\SL(2,\R)$ such that $\bar\alpha=\psi$, i.e.\ $\psi$ is induced by $\alpha$. Now as $R(\alpha)=\infty$ and $D$ is finite, by 
\Cref{inv}\,(4), $R(\psi)=\infty$. Since this holds for all $\psi\in\Aut(\PSL(2,\R))$, we have that $\PSL(2,\R)$ has topological $R_\infty$-property.

Note that $\PGL(2,\R)$ is the quotient group of $\GL(2,\R)$ modulo its center $\{r I\mid r\in\R\mi\{0\}\}$. Then $\PSL(2,\R)$ is a connected 
(normal) characteristic subgroup of index $2$ in $\PGL(2,\R)$. By \Cref{inv}\,(1), $\PGL(2,\R)$ has topological $R_\infty$-property.

\smallskip
\noindent{\bf (5):} Let $n\geq 3$ be such that $n$ is odd. Then the center of $\SL(n,\R)$ is trivial. Let $G:=\widetilde{\SL(n,\R)}$ and let 
$p:G\to \SL(n,\R)$ be the covering map. Then $\ker p$ is the center of $G$ which is invariant under all automorphisms of $G$. Now since 
$G/\ker p$ is isomorphic to $\SL(n,\R)$, which has topological $R_\infty$-property, it follows from \Cref{endo} that $G$ also has this property. 
\end{proof}

\begin{remark} One can take any central subgroup $H$ in $\widetilde{\SL(2,\R)}$ then $H$ is of the form $H_n:=n\Z$, $n\in\N$,  
and the groups $G_n:=\widetilde{\SL(2,\R)}/H_n$ have topological $R_\infty$-property. Note that for the center $Z_n$ of $G_n$, $G_n/Z_n$ 
is isomorphic to $\PSL(2,\R)$, which has topological $R_\infty$-property by Corollary {\rm\ref{PSL2}\,(4)}, and so does $G_n$, by Lemma {\rm{\ref{endo}}}. 
Note also that each $G_n$ is a connected semisimple Lie group with finite center of order $n$, and is a covering group of $\PSL(2,\R)$. If $n$ is even, 
then $G_n$ is a covering group of $\SL(2,\R)$. Thus, all the covering groups of $\SL(2,\R)$ as well as those of $\PSL(2,\R)$ have topological 
$R_\infty$-property. 
\end{remark}

It would be interesting to study if $\widetilde{\SL(n,\R)}$, $\PGL(n,\R)$ and $\PSL(n,\R)$ have topological $R_\infty$-property for $n\geq 4$, 
$n$ even; it would enough to prove it for $\PSL(n,\R)$. Note that for $n$ odd, both $\PGL(n,\R)$ and $\PSL(n,\R)$ are isomorphic to $\SL(n,\R)$. 
More generally, it would be interesting to study if a general connected semisimple Lie group has topological $R_\infty$-property; this would help 
in the study of this property for a general connected non-solvable Lie group.

\medskip
\noindent{\bf Acknowledgements:} We would like to thank P.\ Sankaran for valuable correspondence on the conjugacy classes of semisimple Lie 
groups. We would also like to thank T.\ Mubeena whose talk and subsequent discussions during the Indian Women and Mathematics (IWM) 
Annual Conference 2022 introduced the authors to the field. We would like to acknowledge NBHM, DAE, Govt.\ of India for support and 
IISER Pune for local hospitality during the IWM conference. Riddhi Shah would like to acknowledge the ARG-MATRICS grant from ANRF, India 
(ANRF/ARGM/2025/000695/MTR). Ravi Prakash would like to acknowledge the CSIR-JRF research fellowship
from CSIR, Govt.\ of India.

\end{document}